\documentclass[12pt, reqno]{amsart}
\usepackage{amsmath, amsthm, amscd, amsfonts, amssymb, graphicx, color}
\usepackage[bookmarksnumbered, colorlinks, plainpages]{hyperref}

\newtheorem{theorem}{Theorem}[section]
\newtheorem{lemma}[theorem]{Lemma}
\newtheorem{proposition}[theorem]{Proposition}
\newtheorem{corollary}[theorem]{Corollary}
\theoremstyle{definition}
\newtheorem{definition}[theorem]{Definition}
\newtheorem{example}[theorem]{Example}

\theoremstyle{remark}
\newtheorem{remark}[theorem]{Remark}
\numberwithin{equation}{section}

\begin{document}
\setcounter{page}{1}

\title[NORM-ATTAINABILITY IN NORMED SPACES]{VARIOUS NOTIONS OF NORM-ATTAINABILITY IN NORMED SPACES}
\author[N. B. Okelo]{N. B. Okelo}

\address{Department of Pure and Applied Mathematics\\ School of Mathematics and Actuarial Science\\ Jaramogi Oginga Odinga  University of Science and Technology\\ Box 210-40601, Bondo-Kenya.}
\email{bnyaare@yahoo.com}

\subjclass[2010]{Primary 47B47; Secondary 47A30.}


\date{Received: xxxxxx; Revised: yyyyyy; Accepted: zzzzzz.
\newline \indent $^{*}$ Corresponding author}

\begin{abstract}
Let $H$ be a reflexive, dense, separable, infinite dimensional complex Hilbert space and let $B(H)$ be the algebra of all bounded linear operators on $H$. In this paper, we carry out characterizations of norm-attainable operators in normed spaces. We give conditions for norm-attainability of linear functionals in Banach spaces, non-power operators on $H$ and elementary operators. Lastly, we characterize a new notion of norm-attainability for power operators in normed spaces.
\end{abstract} \maketitle

\section{Introduction}
Studies on Hilbert space operators have elicited a lot of interest from mathematicians for decades. Characterizations of properties of Hilbert space operators have been done by many mathematicians  over a long period of time with interesting results being obtained. Such properties include norms, numerical ranges, positivity, spectrum, invertibility among others. Norm-attainability is also a property which has been give keen attention. This property still remains very important as it has a lot of open questions which are unanswered particulary when a super class of Hilbert space operators called supraposinormal operators \cite{Oke3} are considered. In \cite{Du} the authors characterized the norm property for elementary operators and gave conditions under which a general elementary operator is norm-attainable. In-depth characterization of norm-attainable operators has also been done in details with considerations given to other properties like orthogonality (see \cite{Oke1} - \cite{Oke5} and the references therein). Regarding derivations, authors in \cite{Oke5}  showed that if  $V_{\Gamma}$ and $W_{\Gamma}$ are $\Gamma$-Banach algebras and $\delta$  an
$\alpha$-inner derivation, then $\delta$ is norm-attainable if and only if the adjoint, $\delta^{\ast}$, of $\delta$ is norm-attainable. Moreover, as a consequence they proved that if $\delta_{N}^{1}$ and $\delta_{N}^{2}$ are norm-attainable  then
$\delta_{N}$ is norm attainable if either $\delta_{N}^{1}$ and $\delta_{N}^{2}$ or both are zero derivations
  and $\delta_{N}^{1},$ and $\delta_{N}^{2}$ are $\alpha$-inner derivation and
 $\alpha^{\prime}$-inner derivation respectively. On elementary operators,  necessary and sufficient conditions for norm-attainability for Hilbert space operators were given in \cite{Oke2} where it was proved that if $S\in B(H),\;\beta\in W_{0}(S)$ and $\alpha >0,$ then there exists an operator $Z\in B(H) $ such that $\|S\|=\|Z\|,$ with $\|S-Z\|<\alpha$. Furthermore, there exists a unit vector $\eta\in H$ such that $\|Z\eta\|=\|Z\|$ with $\langle Z\eta,\eta\rangle =\beta,$ where $W_{0}(S)$ denotes the maximal numerical range of the operator $S.$ Moreover, norm-attainability conditions for elementary operators and generalized derivations have been given. For orthogonality of elementary operators in norm-attainable classes, a detailed exposition has been given in \cite{Oke4} and \cite{Oke6}. A superclass of Hilbert space operators have also been considered in norm-attainable classes. In \cite{Oke3}, a good characterization has been done on $\alpha$-supraposinormality of operators in dense norm-attainable classes. In this paper,  we continue in the spirit of characterization of operators in normed spaces. We characterize norm-attainability for functionals in Banach spaces. Moreover, we give a new notion of norm-attainability for power operators and also for elementary operators which generalizes the results of \cite{Oke2}.

\section{Preliminaries}
In this section, we outline preliminary concepts which are useful in the sequel.
\begin{definition}(\cite{Kry})
Let $V$ be a linear vector space. A non-negative real valued function $\|.\|: V \rightarrow \mathbb{R}$
is called a norm on $V$ if  it satisfies the following condtions:
\begin{itemize}
  \item [(i).] $\|x\|\geq 0$ and $\|x\|= 0,$ if and only if $x=0,$ for all $x\in V.$
  \item [(ii).] $\|\alpha x\|=|\alpha|\|x\|,$ for all $x\in V$ and $\alpha\in \mathbb{K}.$
  \item [(iii).] $\|x+y\|\leq\|x\|+\|y\|,$ for all $x, y\in V.$
\end{itemize}
The ordered pair $(V, \|.\|)$ is called a normed linear space or simply a normed space.
\end{definition}

\begin{remark}
Examples of normed spaces are:
Banach space, Hilbert space, Hardy space, Orlicz space, $B(H)$ among others.
\end{remark}
\begin{definition}(\cite{Kry})
An operator $A \in B(H)$ is called a scalar operator
of order $m$ if it possesses a spectral distribution of order $m$, i.e., if there is a
continuous unital morphism $\phi : C^{m}
_{0} (\mathbb{C})\rightarrow B(H)$ such that $\phi(z) = A,$ where
$z$ stands for the identity function on $\mathcal{C}$ and $C^{m}
_{0} (\mathbb{C})$ for the space of compactly
supported functions on $\mathbb{C}$ continuously differentiable of order $m$, $0 \leq m \leq \infty.$
An operator $A_{0} \in B(H)$ is called subscalar if it is similar to the restriction of a
scalar operator to an invariant subspace.
\end{definition}

\begin{definition}(\cite{Oke1})
An operator $A\in B(H)$ is said to be
normal if $AA^{*}=A^{*}A$ and $p$-normal if $A^{p}A^{*}=A^{*}A^{p}$; self-adjoint if $A=A^{*}$; positive if $A=A^{*}$ and
$\langle Ax, x\rangle \geq 0,$ for all $x \in H$; and projection if $A^{2}=A=A^{*}.$
\end{definition}
\begin{definition}\label{NA}(\cite{Oke2}, Definition 1.1) An operator  $A\in B(H)$ is said to be norm-attainable
if there exists a unit vector $x_{0}\in H$ such that $\|Ax_{0}\|=\|A\|.$ The set of all norm-attainable operators on a Hilbert space $H$ is denoted by $NA(H).$
\end{definition}

\begin{definition} (\cite{Oke2}, Definition 1.2) For an operator  $A\in B(H)$ we define a numerical range by $W(A)=\{\langle Ax, x\rangle:\; x\in H,\; \|x\|=1\}$ and the maximal numerical range by $W_{0}(A)=\{\beta\in \mathbb{C}:\; \langle Ax_{n},x_{n}\rangle \rightarrow\beta,\; \text{where}\;\|x_{n}\|=1,\; \|Ax_{n}\|\rightarrow \|A\|\}.$
\end{definition}

\begin{definition}(\cite{Dun}) Let $\Omega$ be a Banach lattice then $\Omega$ is an abstract $M$ space i.e. $\Omega\in AM$ if $x\wedge y=0$ implies $\|x+y\|=\max\{\|x\|,\; \|y\|\}$. Also $\Omega$ is abstract $L$ space i.e. $\Omega\in AL$ if $x\vee y=0$ implies $\|x+y\|=\|x\|+\|y\|$.
\end{definition}
\begin{definition} (\cite{Kry})Let $\Omega$ be a Banach lattice. Then
 $\Omega\in AM$ implies $\Omega^{*}\in AL$ and $\Omega\in AL$ implies $\Omega^{*}\in AM$.
   $\Omega\in AM$ if and only if for any  $x,\;y \in\Omega,x,y\geq0\; $ implies $\|x\vee y\|=\max\{\|x\|,\;\|y\|\}$.
   $\Omega\in AL$ if and only if for any  $x,\;y\in \Omega,x,y\geq0\; $ implies $\|x+y\|=\|x\|+\|y\|$.
\end{definition}
\begin{definition}(\cite{Dun})
Let $\Omega$ be a Banach lattice. Then
$\Omega$ is said to be $\sigma$ complete, if for every order bounded sequence $\{x_{n}\}\in \Omega, \vee_{n\geq1}x_{n}$ exists in $\Omega$.
  Also  $\Omega$ is said to be bounded $\sigma$ complete, provided that the any norm bounded and order monotone sequence in $\Omega$ is order convergent.
\end{definition}
\begin{definition} (\cite{Dun})Let $\Omega$ be a Banach space. An element $x\in D(\Omega)$ is called an extreme point of $B(\Omega)$ if $x=\lambda y+(1-\lambda)z,\; y,\; z\in B(\Omega)$ and $\lambda \in (0,\;1)$, imply $y=z$. In this case, we write $x\in ext B(\Omega)$.
\end{definition}
\begin{definition}\label{ELEM}(\cite{Oke1})
Let $\mathcal{A}$  be a Banach algebra and consider
$T:\mathcal{A}\rightarrow \mathcal{A}.$ The operator $T$ is called an
elementary operator if it has the
representation
 $ T(X)=\sum_{i=1}^{n}A_{i}XB_{i},\;\forall\;X\in \mathcal{A},$
 where $A_{i},\;B_{i}$ are fixed in $\mathcal{A}$ or
$\mathcal{M}(\mathcal{A}),$ where $\mathcal{M}(\mathcal{A})$ is the
multiplier algebra of $\mathcal{A}.$
\end{definition}
\begin{example}
Let  $\mathcal{A}=B(H)$. For
$A,\,B\in B(H)$ we define the particular elementary operators:
\begin{enumerate}
 \item [(i).] the left multiplication operator $L_{A}:B(H)\rightarrow B(H)$ by $L_{A}(X)=AX,$ for all $X\in B(H).$
 \item [(ii).] the right multiplication operator $R_{B}:B(H)\rightarrow B(H)$ by $R_{B}(X)=XB,$ for all $X\in B(H).$
 \item [(iii).] the generalized derivation (implemented by $A,\;B$) by $\delta_{A,B}=L_{A}-R_{B},$ for all $X\in B(H).$
 \item [(iv).] the basic elementary  operator (implemented by $A,\;B$) by $M_{A,\;B}(X)=AXB,$ for all $X\in B(H).$
 \item [(v).] the Jordan elementary  operator (implemented by $A,\;B$) by $\mathcal{U}_{A,\;B}(X)=AXB+BXA,$ for all $X\in B(H).$
\end{enumerate}
\end{example}
\section{Norm-attainability for functionals}
In this section, we characterize norm-attainability of functionals in Banach spaces. We regard $H^{*}$ the dual space of a Hilbert space $H$ to be non-zero throughout this section unless otherwise stated. Let $\varphi \in H^{*}$. Then $\varphi$  is said to be norm-attainable at $\frac{\varphi}{\|\varphi\|}$ if there exists  $T\in B(H)$ such that $\langle \varphi , T\rangle =\|\varphi\|.\|T\|>0.$ $\frac{\varphi^{*}}{\|\varphi^{*}\|}$ is called a support for $\varphi.$
The following proposition shows that any functional is norm-attainable in non-zero dual spaces.
\begin{proposition}\label{P1}
Let $B(W)$ be the set of all bounded linear maps on an Orlicz space $W$ then every $\varphi\in H_{+}^{*}$ is norm-attainable on $B(W)$.
\end{proposition}

\begin{proof}
Let $W$ be a bounded linear Orlicz space and $b_{n}\in W$ such that $b_{n}$ is monotone decreasing to $0$. Given $J_{n}\in B(W)$ we have $\varphi(J_{n})> \|\varphi\|-b_{n}$ since $\pi(J_{n})\leq 1<\infty$ and $\varphi(W_{0})=\{0\}$. Suppose that $\pi (J_{n})\leq 2^{-n}$ and let $J(t)=\sup|J_{n}(t)|$. Then $\pi(J)\leq \sum_{n=1}^{\infty}\pi(J_{n})\leq 1,$ that is, $J\in B(W)$ and $\varphi(J)\geq \sup_{n}(|J_{n}|)=\|\varphi\|,$ because $\varphi\in H^{*}_{+}$.
\end{proof}

\begin{theorem}\label{P2}
Let $B(W)$ be the set of all bounded linear maps on an Orlicz space $W$ then $\varphi\in H^{*}$ is norm-attainable on $B(W)$ if and only there exists $L$ in a subspace $\mathcal{C}$  of $B(W)$ such that $\varphi^{+}=\varphi|_{L}$ and $\varphi^{-}=\varphi|_{G \setminus L}$.
\end{theorem}

\begin{proof}
 It is known that we have  $P, Q\in B(W)$ such that $\varphi|_{L}(P)\; \varphi^{*}\;\|\varphi^{+}\|$ and $ \varphi|_{G\setminus L}(Q)=\|\varphi^{-}\|$ by Proposition \ref{P1}. Assume  that $\pi(P|_{L})\leq \frac{1}{2}$ and $\pi(Q|_{G\setminus L})\leq \frac{1}{2}$.  Then $\pi(P|_{L}-Q|_{G\setminus L})\leq \pi(P)+\pi(Q)\leq 1.$ Indeed, $P|_{L}-Q|_{G\setminus L}\in L(W)$ and $\varphi(P|_{L}-y|_{G\setminus L})=\|\varphi^{+}\|+\|\varphi^{-}\|=\|\varphi\|$. Suppose that $P\in B(W)$ satisfies $\varphi(P)=\|\varphi\|$. Let $L=\{\l\in G:P(l)\geq 0\}$ then we show that $\varphi|_{L}, \; -\varphi|_{G\setminus L}\in H^{*}_{+}$. Now, if $\varphi|_{L}\in H^{*}_{+}$, then there exists $Q\in W^{+}$. If $\varphi|_{L\notin H^{'}}$ then there exists $Q\in W^{+}$ such that $\varphi|_{L}(Q)<0$. Now $\varphi$ being singular,  assume $\pi(Q)\leq \frac{1}{2}$ and $\pi(P)\leq \frac{1}{2}$. Then $J=P|_{G \setminus L}-Q|_{L}\in B(W)$ and so, $\|\varphi^{-}\|\geq \varphi^{-}(-J)=-\varphi^{+}(J)+\varphi(J)\geq \varphi(J)=\varphi(P|_{G\setminus L}-\varphi|_{L}(Q))>\varphi(P|_{G\setminus L})$. This is contrary to our earlier assumption. Lastly, $\|\varphi\|= \|\varphi^{+}\|+  \|\varphi^{-}\|>\varphi|(P|_{L})+\varphi(P|_{G\setminus L})= \varphi|(P)= \|\varphi\|.$  $(\varphi |_{G\setminus L})\in H_{1}^{*}$ follows analogously.
\end{proof}

\begin{corollary}
Let $\varphi\in H^{*}$ is norm-attainable at $P\in B(W)$, then $\varphi(P|_{A})\; \|\varphi|_{A}\|$ for all $A\in \mathcal{C}$.
\end{corollary}

\begin{proof}
 $\|\varphi\|=\|\varphi|_{A}\|+\|\varphi|_{G\setminus A}\|\geq \varphi|_{A}(P)+\varphi|_{G\setminus A}(P)=\varphi(P)=\|\varphi\|$ is enough.
\end{proof}

\begin{theorem}\label{THMM2}
Let $\varphi\in H^{*}$ be singular then the set of all such $\varphi$ is dense in $H^{*}$.
\end{theorem}
\begin{proof}
Let  $\varphi\in H^{*}$ be singular and $\epsilon >0$ be given. Then we have $L\in \mathcal{C}$ such that $\|\varphi^{+}|_{G\setminus L}\|< \epsilon$ and $\|\varphi^{-}|_{L}\|\leq \epsilon$. Suppose that $\psi=\varphi^{+}|_{L}-\varphi^{-}|_{G\setminus L}$. Then by Theorem \ref{P2} $\psi$ is norm-attainable. Also, $\|\varphi-\psi\|\leq \|\varphi^{+}-\psi|_{l}\|+\|\varphi-\psi|_{G\setminus l}\|=\|\varphi^{+}|_{G\setminus l}\|+\|\varphi^{-}|_{l}\|<2 \epsilon.$
\end{proof}

\begin{theorem}\label{thm3}
Let $B(W)$ be the set of all bounded linear maps on an Orlicz space $W$. Then $\phi=\chi+\varphi\;(0\neq \chi\in H_{0}^{*}, \; \varphi\in H^{*})$ is norm-attainable at $P\in B(W)$ if and only if $\pi(P)=1$, $\varphi(P)=\|\varphi\|$ and $\int_{G}k\chi(t)P(t)dt=\pi(x)+\omega(k\chi)$, where $k\in K_{N}(\chi)=\{k:k^{-1}[1+\omega(k\chi)]=\|\chi\|^{0}_{N}\}$.

\end{theorem}
\begin{proof}
From the statement of the theorem we have $\|\phi\|^{0}=f(P)=k^{-1}\langle k\chi,\; \chi\rangle +\varphi(P)\leq k^{-1}[\pi(P)+\omega(k\chi)]+\varphi(P)\leq k^{-1}[1+\omega(k\chi)]+\|\varphi\|=\|\chi\|^{0}_{N}+\|\phi\|=\|\phi\|^{0}.$ The converse follows from the fact that $\phi$ is singular from Theorem \ref{THMM2} and an assertion from Proposition \ref{P1}.
\end{proof}

 At this point, we consider norm-attainable functionals in Banach lattices. We denote an abstract $L$ space  and abstract $M$ space by $AL$ and $AM$ respectively. For details on $AL$ and $AM$  see \cite{Dun}. We state the following lemma.
\begin{lemma} \label{AB1}
Let $\Omega \in AL$ and $\varphi\in B(\Omega^{*})$. Then the following are equivalent.
\begin{itemize}
  \item [(i).] $\varphi$ is norm-attainable.
  \item [(ii).] Both $\varphi^{+}$ and $\varphi^{-}$ are norm-attainable.
  \item [(iii).] $\varphi^{+}$ or $\varphi^{-}$   is norm one.
\end{itemize}
\end{lemma}
\begin{proof} $(i) \Rightarrow (ii).$ Choose $x\in B(\Omega)$ such that $\varphi(x)=\|\varphi\|=1$. From
\begin{eqnarray*}
  1=\|\varphi\| &=& \varphi(x)=\varphi^{+}(x^{+})+\varphi^{-}(x^{-})-\varphi^{+}(x^{-})-\varphi^{-}(x^{+}) \\
   &=& \|\varphi^{+}\|\|x^{+}\|+\|\varphi^{-}\|\|x^{-}\|-\varphi^{+}(x^{-})-\varphi^{-}(x^{+}) \\
   &=& \|\varphi\|(\|x^{+}\|+\|x^{-}\|)=\|\varphi\|\|x\|=\|\varphi\|=1
\end{eqnarray*}
we obtain $\varphi^{+}(x^{+})=\|\varphi^{+}\|\|x^{+}\|;\;\;\; \varphi^{-}(x^{-})=\|\varphi^{-}\|\|x^{-}\|$ and $\varphi^{+}(x^{-})=\varphi^{-}(x^{+})=0$ since $\varphi^{\pm}(x^{\pm})\leq \|\varphi^{\pm}\|\|x^{\pm}\|$ and $\varphi^{\pm},\; (x^{\pm})$ are non-negative.\\
$(ii) \Rightarrow (iii).$ This follows obviously.\\
$(iii) \Rightarrow  (i).$ We suppose that $\varphi^{+}$ is norm one and norm-attainable. choose $x\in B(\Omega^{+})$ such that $\varphi^{+}(x)=\|\varphi\|=1$. We have $\varphi^{-}(x)=0$. Indeed, $1=\|\varphi\| \geq |\varphi|(|x|)\geq \varphi^{+}(x)+\varphi^{-}(|x|)=1+\varphi^{-}(|x|)\geq 1,$ which implies that $\varphi^{-}(x^{+})=\varphi^{-}(x^{-})=0$. Hence, $ \varphi(x)=\varphi^{+}(x)=\|\varphi\|=1$.
\end{proof}

\begin{theorem}
Let $\Omega\in AL$ and $\theta\leq \varphi\in B(\Omega^{*})$. Then the following are equivalent.
\begin{itemize}
  \item [(i).] $\varphi$ is norm-attainable.
  \item [(ii).] There exists $\theta\leq x\neq \theta$ such that $\varphi(y)=\|y\|\; \text{for all} \; \overline{E}_{\Omega}$, the norm closure of $E_{\Omega}$ where $E_{\Omega}=\{y\in \Omega: \; \theta\leq y\leq nx,\; \text{for some}\; n>0\}.$
  \item [(iii).] There exists $\theta \neq x\in \Omega^{+}$ such that among $B(\Omega^{*})=\{\psi\in \Omega^{*};\; \|\psi\|=1\}$, $\varphi$ is maximal on $\overline{E}_{\Omega}$.
  \end{itemize}
  \end{theorem}
\begin{proof}
 $(i) \Rightarrow  (ii).$ Choose $x\in B(\Omega)$ satisfying $\varphi(x)=\|x\|=1$. We have $x\in \theta$ for $1=\varphi(x)=\varphi(x^{+})=\varphi(x^{+})-\varphi(^{-})\leq \varphi(^{+})\leq \|x^{+}\|-\|x^{-}\|\leq 1$ by Lemma \ref{AB1} which implies $\|x^{-}\|=0$. Now if we consider $y\in \overline{E}_{\Omega} $, we need to prove that $\varphi(y)=\|y\|$. Since $\varphi$ is continuous, let $y\in \overline{E}_{\Omega}$, i.e $\theta \leq y\leq nx$ for some $n\geq 0$. But since $n=\varphi(nx)=\varphi(nx-y)+\varphi(y)\leq \|nx-y\|+\|y\|=\|nx\|=n$, we have $\varphi(y)=\|y\|$.\\
$(ii) \Rightarrow  (iii).$ Follows trivially from $(ii) \Rightarrow (iii)$ in Lemma \ref{AB1} and abstractness of $AL$. \\
$(iii) \Rightarrow  (i).$ Let $\varphi$ be maximal in $B(\Omega^{*})$ on $\overline{E}_{\Omega}$ for some $\theta\neq x\in \Omega^{+}$. Choose$\psi\in B(\Omega_{*})$ such that $\psi(x)=\|x\|$, then by $\varphi(x)\geq \psi(x)=\|x\|,$ it is clear that $\varphi$ is norm-attainable at $x/ \|x\|$.
\end{proof}

\begin{proposition}
 Let $\Omega\in AM$ be $\sigma$-complete and $\varphi\in \Omega^{*}$. Then for any $\epsilon> 0$, there exists a subspace $E$ of $\Omega=E+E^{\perp}$ and $\|\varphi^{+}|_{E^{\perp}}\|<\epsilon,\; \|\varphi^{-}|_{E}\|<\epsilon$.
\end{proposition}
\begin{proof}
Choose $x\in D(\Omega)$ with the property that $\varphi^{+}(x)>\|\varphi\|-\epsilon$, and $E=(x^{\perp})^{\perp}$. Then $x^{+}\in E, \; x^{-}\in E^{\perp}$ and $\Omega=E+E^{\perp}$. Furthermore, for $x\in D(\Omega)$ we have that
$  \|\varphi^{+}|_{E}\|+\|\varphi^{+}|_{E^{\perp}}\|+\|\varphi^{-}|_{E}\|+\|\varphi^{-}|_{E^{\perp}}\| = \|\varphi^{+}|\|+\|\varphi^{+}\|=\|\varphi\|<\varphi(x)+\epsilon
  = \varphi^{+}|_{E}(x)+\varphi^{+}|_{E^{\perp}}(x) -\varphi^{-}|_{E}(x)
   -\varphi^{+}|_{E^{\perp}}(x)+\epsilon
$
because $\varphi^{+}|_{E^{\perp}}(x)\leq 0$ and $\varphi^{-}|_{E}(x)\geq 0$. So, we conclude that
 $ \|\varphi^{+}|_{E^{\perp}}\|+\|\varphi^{-}|_{E}\| =\|\varphi^{+}|\|-\|\varphi^{+}|_{E}\|+\|\varphi^{-}|\|-\|\varphi^{-}|_{E}\|
   \leq\|\varphi^{+}\|-\varphi^{+}|_{E}(x)+ \|\varphi^{-}\|-\varphi^{-}|_{E^{\perp}}(x)
   <\varphi^{+}|_{E^{\perp}}(x)-\varphi^{-}|_{E}(x)+\epsilon\leq \epsilon.
$
This completes the proof as required.
\end{proof}

\begin{theorem} \label{thm4}
Let a Banach lattice $\Omega$ be bounded $\sigma$-complete and $B(\Omega)$  order-closed, then every positive bounded linear $\varphi\in \Omega^{*}$ is norm-attainable.
\end{theorem}
\begin{proof}
Consider $x_{n}(\geq 0)\in D(\Omega)$ such that $\varphi(x_{n})\rightarrow \|\varphi\|$. Since $\Omega$ is bounded $\sigma$-complete and $B(\Omega)$ is closed under order, $y=\vee_{n}(x_{n})$ exists in $\Omega$ and $\|y\|=1$. Hence, $y\geq x_{n}\geq 0$ and $\varphi\geq 0$ implies $\|\varphi\|\geq \varphi(y)\geq \varphi(x_{n})\rightarrow \|\varphi\|.$ So,  $x\in D(\Omega)$  exists which satisfies the norm-attainability condition, $\varphi(x)=\|\varphi\|,$ for functionals and this completes the proof.
\end{proof}

\begin{theorem}
Let $\Omega\in AM$ be  bounded $\sigma$-complete and $B(\Omega)$ order-closed, then $\varphi\in \Omega^{*}$ is norm-attainable if and only if there exists a subspace $E$ of $\Omega$ such that $\varphi^{+}=\varphi|_{E}, \; \varphi^{-}=-\varphi|_{E^{\perp}}$.
\end{theorem}
\begin{proof}
\emph{Necessity.} Let $x\in B(\Omega)$ be having the property that $\varphi|(x)=\|\varphi\|$, and define $E=(x^{-})^{\perp}$. Then $\Omega=E+E^{\perp}$ and $x^{+}\in E, \; x^{-}\in E^{\perp}$. Now, $\|\varphi\|=\|\varphi|_{E}\|+\|\varphi|_{E^{\perp}}\|$; we need to show that $\varphi^{+}=\varphi|_{E}$ and $\varphi^{-}=-\varphi|_{E^{\perp}}$, it is enough that $\varphi|_{E}\geq 0$ and  $-\varphi|_{E^{\perp}}\geq 0$. Indeed, if $\varphi_{E}(y)<0$ for some $y(\geq 0)\in D(\Omega)$ then we let $y\in E$. So, $z=-x^{-}-y$ satisfies $\|z\|=\max\{\|x\|,\; \|y\|\}=1$ and hence,
$
  \|\varphi^{-}\| = \varphi^{-}(-z)=\varphi(z)-\varphi^{+}(z)\geq \varphi(z)
 = \varphi|_{E^{\perp}}(-x^{-})-\varphi|_{E}(y)>\varphi|_{E^{\perp}}(-x^{-})=-\varphi|_{E^{\perp}}(-x).
$
Now since $\|\varphi^{+}\|\geq \varphi (x|_{E})=\varphi|_{E}(x)$, this is contrary to
$ \|\varphi\| = \|\varphi^{+}\|+\|\varphi^{-}\|>\varphi|_{E}(x)-\varphi|_{E^{\perp}}(x)= \varphi(x)=\|\varphi\|$. Also $-\varphi|_{E^{\perp}}\geq 0$ follows analogously.\\
\emph{Sufficiency.} We know that there exists $x,y(\geq 0)\in D(\Omega)$ such that $\varphi^{+}(x)=\|\varphi^{+}\|$ and $\varphi^{-}(x)=\|\varphi^{-}\|$ from Theorem \ref{thm4}. Since $\varphi^{+}=\varphi|_{E}$ and $\varphi^{-}=-\varphi|_{E^{\perp}}$, we let $x\in E$ and $y\in E^{\perp}$ and $u\in x-y$. Then $u=\|x-y\|=\max\{\|x\|,\|y\|\}=1$ and so we have
$ \|\varphi\| = \|\varphi^{+}\|+\|\varphi^{-}\|=\varphi^{+}(x)+\varphi^{-}(y)
   = \varphi|_{E}(x)+\varphi|_{E^{\perp}}(-y)=\varphi(u).
$
\end{proof}

\begin{corollary}
Let $\Omega\in AM$ be a $\sigma$-complete and $\varphi\in D(\Omega^{*})$. Then $\varphi \in ext B(\Omega^{*})$ if and only if $\varphi (x)\varphi (y)=0$ for all $x,y\in \Omega$ satisfying $x\wedge y=0$.
\end{corollary}
\begin{proof}
\emph{Necessity.} If there exists $x,y\in \Omega $ satisfying $x\wedge y=0$ but $\varphi(x)>0$ and $\varphi(y)>0$, then we set $E=y^{\perp}$, and $\Omega=E+E^{\perp}$. Let $\psi=\varphi|_{E}$ and $\tau=\varphi|_{E^{\perp}}$. Then $\|\psi\|>0,\; \|\psi\|>0$ since $x\in E,\; y\in E^{\perp}$. Therefore, $\varphi=\|\psi\|\frac{\psi}{\|\psi\|}+\|\tau\|\frac{\tau}{\|\tau\|}$ and $\|\psi\|+\|\tau\|=\|\varphi\|=1.$ Hence, $\varphi\in ext B(\Omega)$.
\\
\emph{Sufficiency.} First we show $\|\varphi^{+}\|\|\varphi^{-}\|=0$. In fact, for any $\epsilon > 0$, by Theorem \ref{thm4}, there exists two orthogonal subspaces $E,\; F\in \Omega$ such that $\Omega=E+F$ and $\|\varphi^{-}|_{E}\|< \epsilon,\;  \|\varphi^{+}|_{F}\|<\epsilon $. Choose $x\in B(\Omega)$ such that $\varphi(x)> \|\varphi\|-\epsilon$, and let $x=u+v$, where $u\in E$ and $v\in F$. Then $\varphi(u)\varphi(v)=0$ since $u\wedge v=0$. If $\varphi(v)=0$ then $\|\varphi\|-\epsilon< \varphi(x)=\varphi^{+}|_{E}(u)-\varphi^{-}|_{E}(u)\leq \|\varphi^{+}|_{E}\|+\|\varphi^{-}|_{E}\|<\|\varphi^{+}\|+\epsilon$. Let $\epsilon \rightarrow 0$, we find $\|\varphi^{-}\|=\|\varphi\|-\|\varphi^{+}\|=0$. Similarly, if $\varphi(u)=0$. Then $\|\varphi^{+}\|=0$. Hence, without loss of generality, we assume $\varphi=\varphi^{+}$. Let $\psi, \; \tau\in D(\Omega^{*})$ satisfy $2\varphi=\psi+\tau$. Then $2\varphi=(\psi^{+}+\tau^{+})-(\psi^{-}+\tau^{-})$ and hence
       $ \|2\varphi\| = \|\psi^{+}\|+\|\tau^{+}\|+\|\psi^{-}\|+\|\tau^{-}\|
        = \|\psi\|+\|\tau\|=2=\|2\varphi\|
      $
Thus $\psi^{+}+\tau^{+}=2\varphi$ and $\psi^{-}=\tau^{-}=0$.
Now we show $\psi=\tau=\varphi$, that is, $\varphi\in ext B(\Omega^{*})$. This follows if we prove that $\psi(y)=\tau(y)=0$ whenever $\varphi(y)=0$ this means $\varphi=a\psi=b\tau$, but $\varphi,\; \psi, \tau\in D(\Omega^{*})$ and $2\varphi=\psi+\tau$, so $a=b=1$, this means assume $y\geq 0$; then from $\psi(y)\geq 0, \tau(y)\geq 0$ and $\psi(y+\tau(y)=2\varphi(y)=0$. We have $\psi(y)=\tau(y)=0$. For the general case, since $\varphi(y)=0$ and by the condition given in theorem $\varphi(y^{+})\varphi (y^{-})=0$, we have $\varphi(y^{+})=\varphi (y^{-})=0$ hence the condition $\psi(y)=\tau(y)=0$ follows from the first case.

\end{proof}

\section{Norm-attainability for non-power operators}
This section deals with norm-attainability for usual Hilbert space operators. To avoid ambiguity and without loss of generality, we use the term "non-power" to differentiate these operators from those in the next section. The first result is a proof of norm-attainability condition for compact self adjoint operators in a dense separable infinite dimensional complex Hilbert space.
\begin{theorem} Let $H$ be a dense separable infinite dimensional complex Hilbert space. Let $T: H\rightarrow H$ be a compact and self-adjoint operator on Hilbert space $H$. Let $D$ be unit sphere in $H$. Then there exists a vector $x_{0}\in D$ such that $\|Tx_{0}\|=\|T\|$.
\end{theorem}
\begin{proof} From an analogy of the definition of the usual operator norm, we know that $\|T\|=\sup_{x\in D}\|Tx\|$, so there exists a sequence of elements  $x_{1},x_{2},...\in D$ such that $\lim_{n\rightarrow \infty}\|Tx_{n}\|=\|T\|$. Since $T$ let
$
y_{0}=\lim_{n\rightarrow \infty}Tx_{n}\; \text{exists in}\; H.
$
Let $Y=\overline{Span\{x_{1}, \;x_{2},...\}}$. Then it is a closed subspace of $H$, hence a Hilbert space. So, $Y$ is reflexive and separable. Thus, there exists a subsequence of $(x_{n})$, say $(x_{j_{1}},x_{j_{2}},x_{j_{3}},...)$  with $1\leq j_{1}<j_{3}<...,$ such that $x_{j_{k}}$ converge weakly in $Y$to some point $x_{0}\in Y$. Now, for each $z\in Y$,
    $\langle Tx_{0}, z\rangle=\langle x_{0}, Tx\rangle=\lim_{k\rightarrow \infty}\langle x_{j_{k}}, Tz \rangle=\lim_{k\rightarrow}\langle Tx_{x_{k}},z \rangle= \langle y_{0},z\rangle.
$
Now, the second equality holds since $x_{j_{k}}$ converge weakly to $x_{0}$. Indeed, $x\mapsto \langle x, Tz\rangle$ is a bounded linear functional on $Y,$ and the last equality holds because we know from the fact that $y_{0}=\lim_{n\rightarrow \infty}Tx_{n}\; \text{exists in}\; H.$  So, $\lim_{k\rightarrow \infty}Tx_{j_{k}}=y_{0}$ exists and $\langle., .\rangle$ is continuous. Since $\langle Tx_{0}, z\rangle=\langle x_{0}, Tx\rangle=\lim_{k\rightarrow \infty}\langle x_{j_{k}}, Tz \rangle=\lim_{k\rightarrow}\langle Tx_{x_{k}},z \rangle= \langle y_{0},z\rangle$ holds for $z\in Y,$ we conclude that $Tx_{0}=y_{0}$. Hence, $\|Tx_{0}\|=\|y_{0}\|=\|\lim_{k\rightarrow \infty}T(x_{n})\|=lim_{k\rightarrow \infty}\|T(x_{n})\|=\|T\|$. Now, since $x_{j_{k}}$ converge weakly to $x_{0}$ we have $\langle x_{0}, x_{0}\rangle=\lim_{k\rightarrow \infty}\langle x_{j_{k}}, x_{0}\rangle$ and $|\langle x_{j_{k}}, x_{0}\rangle|\leq \|x_{j_{k}}\|.\|x_{0}\|=1$ for all $k$, hence $\|x_{0}\|\leq 1$. We cannot have $\|x_{0}\|<1$ since then $\|Tx_{0}\|\leq \|T\|.\|x_{0}\|<\|T\|$ which is a contradiction. Hence $\|x_{0}\|=1$ i.e $x_{0}\in D$. Hence, the existence of  $x_{0}$ is proved as required.
\end{proof}
\begin{corollary} Let $H$ be a dense separable infinite dimensional complex Hilbert space. Let $T: H\rightarrow H$ be a compact and self-adjoint operator on Hilbert space $H$ and let $Y\subset H$ be a subspace such that $T(Y)\subset Y$. Let $T(Y^{\perp})\subset Y^{\perp}$ be  a subspace such that $T(Y)\subset Y$. Then $T(Y^{\perp})\subset Y^{\perp}$, and $T|_{Y^{\perp}}: Y^{\perp}\rightarrow Y^{\perp}$ is a bounded self-adjoint linear operator on Hilbert space $Y^{\perp}$, with the norm $\|T|_{Y^{\perp}}\|\leq \|T\|$.
\end{corollary}
\begin{proof} Choose $z\in Y^{\perp}$. Now, for each $y\in Y, \langle Tz, y\rangle= \langle z, Ty\rangle=0$ since $Ty\in T(Y)\subset Y$ and $z\in Y^{\perp}$. Clearly, $Tz\in Y^{\perp}$. Hence, $T(Y^{\perp})\subset Y^{\perp}$.
\end{proof}

The next proposition in this section is a very important refined result for a characterization of operators in $B(H)$  whose analogies can be found in \cite{Oke2} Theorem 2.1 or \cite{Oke4} Proposition 2.1. Since this is a refinement from the stated earlier results, we include the proof for completion.
\begin{lemma}\label{pro1}
  Let $H$ be a reflexive, dense, separable, infinite dimensional complex Hilbert space. Let $S\in B(H),\;\beta\in W_{0}(S)$ and $\alpha >0$ be given. Then the there exists a self-adjoint compact operator $Z\in B(H) $ such that $\|S\|=\|Z\|,$ with $\|S-Z\|<\alpha$. Furthermore,
  there exists a unit vector $\eta\in H$ such that $\|Z\eta\|=\|Z\|$ with $\langle Z\eta,\eta\rangle =\beta.$
\end{lemma}

  \begin{proof}
Let $\|S\|=1$ and also that $0<\alpha < 2.$ Let $x_{n}\in H\;(n=1,2,...)$ be such that $\|x_{n}\|=1,$  $\|Sx_{n}\|\rightarrow 1$ and also $\lim_{n\rightarrow\infty}\langle Sx_{n},x_{n}\rangle =\beta.$  Let $S=GL$ be the polar decomposition of $S.$ Here $G$ is a partial isometry and we write $L=\int_{0}^{1}\beta dE_{\beta},$ the spectral decomposition of $L=(S^{*}S)^{\frac{1}{2}}.$ Since $L$ is a positive operator with norm $1,$  for any $x\in H$ we have that $\|Lx_{n}\|\rightarrow 1$ as $n$ tends to $\infty$ and $\lim_{n\rightarrow\infty}\langle Sx_{n},x_{n}\rangle =\lim_{n\rightarrow\infty}\langle GLx_{n},x_{n}\rangle =\lim_{n\rightarrow\infty}\langle Lx_{n},G^{*}x_{n}\rangle.$ Now for $H=\overline{Ran(L)}\oplus KerL,$ we can choose $x_{n}$ such that $x_{n}\in \overline{Ran(L)}$ for large $n.$ Indeed, let $x_{n}=x_{n}^{(1)}\oplus x_{n}^{(2)},\; n=1,2,...$ Then we have that $Lx_{n}=Lx_{n}^{(1)}\oplus Lx_{n}^{(2)}=Lx_{n}^{(1)}$ and that $\lim_{n\rightarrow\infty}\|x_{n}^{(1)}\|=1,\;\lim_{n\rightarrow\infty}\|x_{n}^{(2)}\|=0$ since $\lim_{n\rightarrow\infty}\|Lx_{n}\|=1.$ Replacing $x_{n}$ with $\frac{x_{n}^{(1)}}{\|x_{n}^{(1)}\|},$ we get
$\lim_{n\rightarrow\infty}\left\|L\frac{1}{\|x_{n}^{(1)}\|}x_{n}^{(1)}\right\|=
\lim_{n\rightarrow\infty}\left\|S\frac{1}{\|x_{n}^{(1)}\|}x_{n}^{(1)}\right\|=1,$
$\lim_{n\rightarrow\infty}\left\langle S\frac{1}{\|x_{n}^{(1)}\|}x_{n}^{(1)},
\frac{1}{\|x_{n}^{(1)}\|}x_{n}^{(1)}\right\rangle =\beta
.$
 Next let $x_{n}\in \overline{RanL}.$ Since $G$ is a partial isometry from $ \overline{RanL}$ onto $ \overline{RanS},$ we have that $\|Gx_{n}\|=1$ and $\lim_{n\rightarrow\infty}\langle Lx_{n},G^{*}x_{n}\rangle=\beta.$ Since $L$ is a positive operator, $\|L\|=1$ and for any $x\in H,$ $\langle Lx,x\rangle \leq \langle x,x\rangle =\|x\|^{2}.$ Replacing $x$ with $L^{\frac{1}{2}}x$, we get that $\langle L^{2}x,x\rangle \leq \langle Lx,x\rangle,$ where $L^{\frac{1}{2}}$ is the positive square root of $L.$ Therefore we have that $\|Lx\|^{2}=\langle Lx,Lx\rangle \leq \langle Lx,x\rangle.$ It is obvious that $\lim_{n\rightarrow\infty}\|Lx_{n}\|=1$ and that $\|Lx_{n}\|^{2}\leq\langle Lx_{n},x_{n}\rangle \leq \|Lx_{n}\|^{2}=1.$ Hence,
$\lim_{n\rightarrow\infty}\langle Lx_{n},x_{n}\rangle =1=\|L\|.$ Moreover, Since $I-L\geq 0,$ we have $\lim_{n\rightarrow\infty}\langle (I-L)x_{n},x_{n}\rangle =0.$ thus $\lim_{n\rightarrow\infty}\| (I-L)^{\frac{1}{2}}x_{n}\| =0.$ Indeed, $\lim_{n\rightarrow\infty}\| (I-L)x_{n}\| \leq\lim_{n\rightarrow\infty}\| (I-L)^{\frac{1}{2}}\|.\| (I-L)^{\frac{1}{2}}x_{n}\| =0.$ For  $\alpha >0,$ let $\gamma =[0,1-\frac{\alpha}{2}]$ and let $\rho =(1-\frac{\alpha}{2},1].$ We have
$ L = \int_{\gamma}\mu dE_{\mu}+\int_{\rho}\mu dE_{\mu}
   = LE(\gamma) \oplus LE(\rho).$
Next we show that $\lim_{n\rightarrow\infty}\|E(\gamma)x_{n}\| =0.$ If there exists a subsequence $x_{n_{i}}, (i=1,2,...,)$ such that $\|E(\gamma)x_{n_{i}}\| \geq\epsilon > 0,\;(i=1,2,...,)$, then since $\lim_{i\rightarrow\infty} \|x_{n_{i}}-Lx_{n_{i}}\| =0,$ it follows that
 $  \lim_{i\rightarrow\infty} \|x_{n_{i}}-Lx_{n_{i}}\|^{2} =
       \lim_{i\rightarrow\infty} (\|E(\gamma)x_{n_{i}}-LE(\gamma)x_{n_{i}}\|^{2}
       +\|E(\rho)x_{n_{i}}-LE(\rho)x_{n_{i}}\|^{2})
        = 0.
   $
Hence, we have that $\lim_{i\rightarrow\infty} \|E(\gamma)x_{n_{i}}-LE(\gamma)x_{n_{i}}\|^{2} =0.$ Now, it is clear that
$\|E(\gamma)x_{n_{i}}-LE(\gamma)x_{n_{i}}\| \geq \|E(\gamma)x_{n_{i}}\|-\|LE(\gamma)\|.\|E(\gamma)x_{n_{i}}\|
  \geq (I-\|LE(\gamma)\|)\|E(\gamma)x_{n_{i}}\|
  \geq \frac{\alpha}{2} \epsilon
  > 0.
$
\noindent This is a contradiction. Therefore, $\lim_{n\rightarrow\infty}\|E(\gamma)x_{n}\| =0.$ Since $\lim_{n\rightarrow\infty}\langle Lx_{n},x_{n}\rangle =1,$ we have that
$\lim_{n\rightarrow\infty}\langle LE(\rho)x_{n},E(\rho)x_{n}\rangle =1$ and
$\lim_{n\rightarrow\infty}\langle E(\rho)x_{n},G^{*}E(\rho)x_{n}\rangle =\beta.$\\
\noindent It is easy to see that $\lim_{n\rightarrow\infty}\| E(\rho)x_{n}\| =1,\;\lim_{n\rightarrow\infty}\left\langle L\frac{E(\rho)x_{n}}{\|E(\rho)x_{n}\|},\frac{E(\rho)x_{n}}{\|E(\rho)x_{n}\|} \right\rangle =1$ and
$\lim_{n\rightarrow\infty}\left\langle L\frac{E(\rho)x_{n}}{\|E(\rho)x_{n}\|},G^{*}\frac{E(\rho)x_{n}}{\|E(\rho)x_{n}\|} \right\rangle =\beta$
Replacing $x$ with $\frac{E(\rho)x_{n}}{\|E(\rho)x_{n}\|},$ we can assume that $x_{n}\in E(\rho)H$ for each $n$ and $\|x_{n}\|=1.$ Let
$J =\int_{\gamma}\mu dE_{\mu}+\int_{\rho}\mu dE_{\mu}
   = J_{1} \oplus E(\rho).
$
Then it is evident that $\|J\|=\|S\|=\|L\|=1, Jx_{n}=x_{n},$ and $\|J-L\|\leq\frac{\alpha}{2}.$
 If we can find a contraction $V$ such that $\|V-G\|\leq\frac{\alpha}{2}$ and $\|Vx_{n}\|=1$ and $\langle Vx_{n},x_{n}\rangle =\beta,$ for a large $n$ then letting $Z=VJ$, we have that $\|Zx_{n}\|=\|VJx_{n}\|=1,$ and that $\langle Zx_{n},x_{n}\rangle =\langle VJx_{n},x_{n}\rangle=\langle Vx_{n},x_{n}\rangle =\beta,$
$ \|S-Z\| = \|GL-VJ\|
   \leq \|GL-GJ\|+\|GJ-VJ\|
   \leq \|G\|\cdot\|L-J\|+\|G-V\|\cdot\|J\|
   \leq \frac{\alpha}{2}+\frac{\alpha}{2}
   = \alpha .
$
Lastly, we now construct the desired contraction $V$.
Clearly, $\lim_{n\rightarrow\infty}\langle x_{n},G^{*}x_{n}\rangle =\beta,$ because
$\lim_{n\rightarrow\infty}\langle L x_{n},G^{*}x_{n}\rangle =\beta$ and
$\lim_{n\rightarrow\infty}\| x_{n}-Lx_{n}\| =0.$ Let $Gx_{n}=\phi_{n}x_{n} +\varphi_{n}y_{n},\;\;(y_{n}\bot x_{n},\;\|y_{n}\|=1)$ then $\lim_{n\rightarrow\infty}\phi_{n} =\beta,$ because $\lim_{n\rightarrow\infty}\langle  Gx_{n},x_{n}\rangle =\lim_{n\rightarrow\infty}\langle  x_{n},G^{*}x_{n}\rangle =\beta$ but $\|Gx_{n}\|^{2}=|\phi_{n}|^{2}+|\varphi_{n}|^{2}=1,$ so we have that
$\lim_{n\rightarrow\infty}|\varphi_{n}| =\sqrt{1-|\beta}|^{2}.$ Now
 without loss of generality, there exists an integer $M$ such that $|\phi_{M}-\beta|<\frac{\alpha}{8}.$ Choose $\varphi_{M}^{0}$ such that $|\varphi_{M}^{0}| =\sqrt{1-|\beta}|^{2},\;\;|\varphi_{M}-\varphi_{M}^{0}|<\frac{\alpha}{8}.$ We have that
 $  Gx_{M} = \phi_{M}x_{M}+\varphi_{M}y_{M} -\beta x_{M}+\beta x_{M}-\varphi_{M}^{0}y_{M}+\varphi_{M}^{0}y_{M}
        = (\phi -\beta)x_{M} +(\varphi_{M}-\varphi_{M}^{0})y_{M}+\beta x_{M}+\varphi_{M}^{0}y_{M}.
   $
Let $q_{M}=\beta x_{M}+\varphi_{M}^{0}y_{M},$ $ Gx_{M}=(\phi -\beta)x_{M} +(\varphi_{M}-\varphi_{M}^{0})y_{M}+q_{M}.$ Suppose that $y\bot x_{M},$ then
$\langle Gx_{M}, Gy\rangle = (\phi -\beta)\langle x_{M}, Gy\rangle+(\varphi_{M}-\varphi_{M}^{0})\langle y_{M}, Gy\rangle +\langle q_{M}, Gy\rangle
   = 0,
$
because $G^{*}G$ is a projection from $H$ to $RanL.$ It follows that
$|\langle q_{M}, Gy\rangle|\leq |\phi_{M}-\beta|.\|y\| + |\varphi_{M}-\varphi_{M}^{0}|.\|y\|\leq \frac{\alpha}{4}\|y\|.$ If we suppose that
$Gy=\phi q_{M} + y^{0}, \; (y^{0}\bot q_{M},)$ then $y^{0}$ is uniquely determined by $y.$ Hence we can define $V$ as follows $V:x_{M}\rightarrow q_{M},\;y\rightarrow y^{0},\;\phi x_{M}+\varphi_{M}y \rightarrow \phi q_{M}+\varphi_{M}y^{0},$ with both $\phi ,\varphi$ being complex numbers. $V$ is a linear operator. We prove that $V$ is a contraction. Now, $\|Vx_{M}\|^{2}=\|q_{M}\|^{2}=|\beta|^{2}=|\varphi_{M}^{0}|^{2}=1,$
$\|Vy\|^{2}=\|Gy\|^{2}-|\phi y|^{2}\leq\|Gy\|^{2}\leq \|y\|^{2}.$ It follows that $\|V\phi\|^{2}=\|\phi\|^{2}\|Vx_{M}\|^{2}+|\varphi |^{2}\|Vy\|^{2}\leq|\phi|^{2}+ |\varphi|^{2}=1,$ for each $x\in H$ satisfying that $x=\phi x_{M}+\varphi_{M}y,\;\;\|x\|=1,\;x_{M}\bot y,$ which is equivalent to that $V$ is a contraction. From the definition of $V$, we can show that
$\|Gx_{M}-Vx_{M}\|^{2}=|\phi -\beta|^{2} +|\varphi_{M}-\varphi_{M}^{0}|^{2}\leq \frac{2\alpha^{2}}{16}=\frac{1}{8}\alpha^{2}.$
If $y\bot x_{M},\; \|y\|\leq 1$ then obtain
$\|Gy-Vy\|=|\phi|\|Vx_{M}\|=|\langle Gy,Vx_{M}\rangle|= |\langle q_{M},Gy\rangle|< \frac{\alpha}{4}.$
Hence for any $x\in H,\;x=\phi x_{M}+\varphi_{M}y,\;\;\|x\|=1,$
$ \|Gx-Vx\|^{2} = \|\phi(G-V)x_{M}+\varphi(G-V)y\|^{2}
   = |\phi|^{2}\|(G-V)x_{M}\|^{2}+|\varphi|^{2}\|(G-V)y\|^{2}
    < |\phi|^{2}.\frac{\alpha^{2}}{16}+|\varphi|^{2}.\frac{\alpha^{2}}{16}
    < \frac{\alpha^{2}}{8},
$
which implies that $\|(G-V)x\|<\frac{\alpha}{2},\;\;\|x\|=1,$ and hence $\|(G-V)\|<\frac{\alpha}{2}.$ Let $Z=VJ$. Then $Z$ is what we desire and this completes the proof.
\end{proof}
The following theorem  in \cite{Oke1}  is useful in the sequel. We state it but we omit the proof.

\begin{theorem}\label{Adj}
For $A\in B(H),$ $A$ is norm-attainable if and only if its adjoint is norm-attainable.
\end{theorem}

\begin{remark}
Norm-attainable operators can be expressed in terms of scalar operators or can undergo perturbations but still remain norm-attainable. For instance, if $A, B\in B(H)$ are norm-attainable then
$J=A+iB,$ $A+B$, $A-B$, $\lambda A$, $AI$ are norm-attainable where $\lambda$ is a scalar operator.
\end{remark}

\section{ Norm-attainability for elementary operators}
In this section, we discuss the notion  of norm-attainability  for elementary  operators in Definition \ref{ELEM}.

\begin{proposition}
 Let $H$ be a reflexive, dense, separable, infinite dimensional complex Hilbert space  and $S\in B(H).$  $\delta_{S}$ is norm-attainable if there exists a unit vector $\zeta\in H$ such that $\|S\zeta\|=\|S\|,\;\;\langle S\zeta,\zeta\rangle =0.$
\end{proposition}
\begin{proof}
  Define $X$ by setting $X:\zeta \rightarrow \zeta ,\;\;S\zeta \rightarrow -S\zeta, \;\;x\rightarrow 0,$  whenever $x\bot \{\zeta,S\zeta\}.$ Since $X$ is a bounded operator on $H$ and $\|X\zeta\|=\|X\|=1,$ $\|SX\zeta -XS\zeta\|=\|S\zeta -(-S\zeta)\|=2\|S\zeta\|=2\|S\|.$ It follows that
$\|\delta_{S}\|=2\|S\|$ via the result in \cite{Oke2}, because $\langle S\zeta,\zeta\rangle =0\in W_{0}(S).$ Hence we have that $\|SX-XS\|=2\|S\|=\|\delta_{S}\|.$ Therefore, $\delta_{S}$ is norm-attainable.
\end{proof}
\noindent The next result gives the conditions for norm-attainability of a generalized derivation. We give the following proposition.

\begin{proposition}\label{pro2}
  Let $H$ be a reflexive, dense, separable, infinite dimensional complex Hilbert space. Let $S,T\in B(H).$ If there exists unit vectors $\zeta, \eta\in H$ such that $\|S\zeta\|=\|S\|,\;\|T\eta\|=\|T\|$ and $\frac{1}{\|S\|}\langle S\zeta,\zeta\rangle =- \frac{1}{\|T\|}\langle T\eta,\eta\rangle ,$ then $\delta_{S,T}$ is norm-attainable.
\end{proposition}

\begin{proof}
By linear dependence of vectors, if $\eta$ and $T\eta$ are linearly dependent, i.e.,$T\eta =\phi\|T\|\eta,$ then it is true that $|\phi|=1$ and $|\langle T\eta, \eta\rangle |=\|T\|$. It follows that $|\langle S\zeta, \zeta\rangle |=\|S\|$ which implies that $S\zeta =\varphi\|S\|\zeta$ and $|\varphi|=1.$ Hence $\left\langle\frac{S\zeta}{\|S\|},\zeta\right\rangle = \varphi =- \left\langle\frac{T\eta}{\|T\|},\eta\right\rangle =-\phi .$ Defining $X$ as $X:\eta \rightarrow\zeta  ,\;\; \{\eta\}^{\bot} \rightarrow 0,$ we have $\|X\|=1$ and
$(SX-XT)\eta =\varphi(\|S\|+\|T\|)\zeta ,$ which implies that $\|SX-XT\|=\|(SX-XT)\eta\|=\|S\|+\|T\|.$ By \cite{Oke2}, it follows that
$\|SX-XT\|=\|S\|+\|T\|=\|\delta_{S,T}\|.$ That is $\delta_{S,T}$ is norm-attainable.
\noindent If $\eta$ and $T\eta$ are linearly independent, then $\left|\left\langle\frac{T\eta}{\|T\|},\eta\right\rangle \right|<1,$ which implies that $\left|\left\langle\frac{S\zeta}{\|S\|},\zeta\right\rangle\right|<1.$ Hence $\zeta$ and $S\zeta$ are also  linearly independent. Let us redefine $X$ as follows:
$X:\eta \rightarrow\zeta  ,\;\;\frac{T\eta}{\|T\|}\rightarrow -\frac{S\zeta}{\|S\|},\;\; x \rightarrow 0,$ where $x\in \{\eta,T\eta\}^{\bot}.$
We show that $X$ is a partial isometry. Let $\frac{T\eta}{\|T\|}=\left\langle\frac{T\eta}{\|T\|},\eta\right\rangle \eta +\tau h,\;\; \|h\|=1,\;h\bot\eta .$
Since $\eta$ and $T\eta$ are linearly independent, $\tau\neq 0.$ So we have that
$X\frac{T\eta}{\|T\|}=\left\langle\frac{T\eta}{\|T\|},\eta\right\rangle X\eta +\tau X h=-\left\langle\frac{S\zeta}{\|S\|},\zeta\right\rangle \zeta +\tau X h,$ which implies that
$\left\langle X\frac{T\eta}{\|T\|},\zeta\right\rangle=-\left\langle\frac{S\zeta}{\|S\|},\zeta\right\rangle +\tau \langle X h,\zeta\rangle =-\left\langle\frac{S\zeta}{\|S\|},\zeta\right\rangle . $\\ It follows then that $\langle X h,\zeta\rangle =0$ i.e., $X h\bot\zeta (\zeta =X\eta).$ Hence we have that $\left\|\left\langle\frac{S\zeta}{\|S\|},\zeta\right\rangle \zeta\right\|^{2}+\|\tau X h\|^{2}=\left\|X\frac{T\eta}{\|T\|}\right\|^{2}
=\left|\left\langle\frac{T\eta}{\|T\|},\eta\right\rangle\right|^{2}+|\tau|^{2}=1,$
which implies that $\|Xh\|=1.$ Now it is evident that $X$ a partial isometry and
$\|(SX-XT)\zeta\|=\|SX-XT\|=\|S\|+\|T\|,$ which is equivalent to $\|\delta_{S,T}(X)\|=\|S\|+\|T\|.$ By a simple calculation and considering and \cite{Oke4}, $\|\delta_{S,T}\|=\|S\|+\|T\|.$ Hence $\delta_{S,T}$ is norm-attainable.
\end{proof}

\begin{corollary}
Let $S,T\in B(H)$ If both $S$ and $T$ are norm-attainable then the basic elementary operator  $M_{S,\;T} $ is also norm-attainable.
\end{corollary}

\begin{proof}
For any pair $(S,T)$ it is known that $\|M_{S,\;T}\|=\|S\|\|T\|.$ We can assume that $\|S\|=\|T\|=1.$ If both $S$ and $T$ are norm-attainable, then there exists unit vectors $\zeta$ and $\eta$ with $\|S\zeta\|=\|T\eta\|=1.$ We can therefore define
an operator $X$ by $X=\langle \cdot ,T\eta\rangle \zeta$. Clearly, $\|X\|=1.$
Therefore, we have $\|SXT\|\geq \|SXT\eta\|=\|\|T\eta\|^{2}S\zeta\|=1.$
Hence, $\|M_{S,\;T}(X)\|=\|SXT\|=1,$ that is $M_{S,\;T} $ is also norm-attainable.
\end{proof}

\begin{proposition}
     If the elementary operator $T_{A_{i}, B_{i}}$ is norm-attainable
     then there exists an isometry or a co-isometry $Q_{0}$ such that
      $\|T_{A_{i}, B_{i}}\|=\|\Sigma _{i=1}^{n}A_{i}Q_{_{0}}B_{i}\|.$
\end{proposition}
\begin{proof}
From \cite{Du},  if the elementary operator
$T_{A_{i}, B_{i}}$ is norm-attainable
then there exists a contractive operator $X_{0}\in B(H)$ such that
$\|T_{A_{i}, B_{i}}\|=\|\Sigma
_{i=1}^{n}A_{i}X_{_{0}}B_{i}\|.$ Lastly, it's sufficient enough to
show that $X_{0}=\frac{1}{2}(\Gamma_{1}+\Gamma_{2})$ for two
isometries or co-isometries $\Gamma_{1},\,\Gamma_{2}\in B(H)$. This
follows immediately from Lemma 3.1 in \cite{Du}.
\end{proof}

\section{New Notion of Norm-attainability for Power Operators}
In this section, we give new notions of norm-attainability for power operators.  First we begin with some auxiliary results.

\begin{proposition}\label{p normattainable}
Let $T\in NA(H).$ Then $T$ is $p$-norm-attainable if it is $p$-normal.
\end{proposition}

\begin{proof}
Let $T\in NA(H)$ be $p$-normal, i.e. $T^{p}T^{*} = T^{*}T^{p}.$ By a simple manipulation, it is easy to see that
$T^{p}(T^{*})^{p} = (T^{*})^{p}T^{p}.$
Hence, $T^{p}$ is normal. Now, if we consider $T^{p}$ to be normal. Then $T^{p}T = TT^{p}$ since,  $T^{*}T^{p} = T^{p}T^{*}$ by by Fuglede property. Therefore, $T$ is $p$-normal. But  $T\in NA(H)$ and $T^{p}$ is normal so it follows that  there exists a unit vector $x\in H$ such that $\|T^{p}x\|=\|T^{p}\|, $ for any $p\in \mathbb{N}.$
\end{proof}
\begin{remark}
Every norm-attainable operator and  every bounded normal operator is is $p$-norm-attainable and $p$-normal respectively for any $p\in \mathbb{N}$. However, the converses are not true in general
\end{remark}

\begin{theorem}
Let $NA_{p}(H)$ be the set of all $p$-norm-attainable operators on $H.$ Then $NA_{p}(H)$ is a closed subset of
$NA(H)$ which is closed under scalar multiplication if and only if for any $T\in NA(H), $ $T$ is $p$-normal.
\end{theorem}
\begin{proof}
Consider  $\pi\in \mathbb{K}$ and let $T$  be $p$-normal. By carrying out a simple calculation we find that $(\pi T)p(\pi T)^{*} =(\pi T)^{*}(\pi T)p.$ and so  $\pi T$ is $p$-normal. If $T  \in NA(H)$ then the converse follows immediately by taking limits over a sequence of vectors in $H.$ Therefore, $T$ is $p$-normal.
\end{proof}
\begin{corollary}
Let $T \in NA(H)$ be $p$-norm-attainable. Then the following conditions hold:

\begin{itemize}
  \item [(i).] $UTU^{*}$ is $p$-norm-attainable for a unitary operator $U$.
  \item [(ii).] $T^{*}$ is $p$-norm-attainable.
  \item [(iii).] $T^{-1}$ is $p$-norm-attainable if it exists.
  \item [(iv).] If there exists a unitary equivalence between $T_{0} \in NA(H)$ and $T$, then $T_{0}$ is $p$-norm-attainable.
  \item [(v).] If $G$ is a unitarily invariant subspace of $H$ such that $G$ reduces $T$, then $T_{0}=T/G$
is a $p$-norm-attainable operator.
\end{itemize}
\end{corollary}

\begin{proof}
$(i).$ Since $U$ is unitary then $UU^{*}=U^{*}U=I$, where $I$ is the identity operator then by Definition \ref{NA} and by Theorem \ref{Adj} the obtain the result.\\
$(ii).$ Since $T$ is $p$-norm-attainable from Proposition \ref{p normattainable}, $T^{p}$ is $p$-norm-attainable and so $(T^{*})^{p}$ $p$-norm-attainable. Consequently, $T^{*}$ is $p$-norm-attainable.\\
$(iii).$ Suppose that $T^{-1}$ exists. Since $T$ is $p$-norm-attainable, $T^{p}$ is $p$-norm-attainable. Since $(T^{p})^{-1} = (T^{-1})^{p}$ is $p$-norm-attainable, $T^{-1}.$ is a $p$-norm-attainable operator.
$(iv).$ Follows immediately from $(i).$\\
$(iv).$Follows from the fact that $G$ is
invariant under $T.$ The rest is clear.
\end{proof}

\begin{theorem}
Let $A, B\in  NA_{p}(H)$ be  commuting $p$-norm-attainable operators, then $AB$ is a
$p$-norm-attainable operator.
\end{theorem}
\begin{proof}
Suppose that  $A, B\in  NA_{p}(H)$. Since $A, B$ are commuting $p$-norm-attainable operators,then  $A^{p},\; B^{p}$ are commuting
normal operator. From \cite{Oke3}, $A^{p}B^{p}$ is a $p$-norm-attainable operator and hence norm-attainable. Lastly, $A^{p}B^{p} = (AB)^{p}= (BA)^{p}$
is normal and norm-attainable. Hence $AB$ is a $p$-norm-attainable operator.
\end{proof}
\begin{remark}
Not all $p$-norm-attainable operators are $p$-normal. In fact, The following example shows that sum of two commuting $p$-normal operators need not be $p$-normal.
\end{remark}

\begin{example}
Consider $A=\left[
           \begin{array}{cc}
             1 & 0 \\
             0 & 1 \\
           \end{array}
         \right]$
         and
         $B=\left[
           \begin{array}{cc}
             0 & 1 \\
             0 & 0 \\
           \end{array}
         \right]$
Clearly, $A$ and $B$ are
commuting $2$-normal operators. Now, $A+B=\left[
           \begin{array}{cc}
             1 & 1\\
             0 & 1 \\
           \end{array}
         \right]$ and
         $(A+B)^{2}=\left[
           \begin{array}{cc}
             1 & 2\\
             0 & 1 \\
           \end{array}
         \right]$is not normal. So $A+B$ is not $2$-normal. Recall that $B$ is  self-adjoint.
\end{example}
\section{Open questions}
In conclusion, we have  characterized norm-attainability for functionals in Banach spaces. Moreover, we have given a new notion of norm-attainability for power operators and also for elementary operators in normed spaces. The following two open questions arise naturally: \\
1. Let $H$ be a reflexive, dense, separable, infinite dimensional complex Hilbert space. Does there exist a bounded self-adjoint
operator $A : H \rightarrow H$ such that  $\|Ax_{0}\|<\|A\|,$ for all $x_{0}\in H$?\\
 2. When does $p$-norm-attainablity and $p$-normality coincide in general Banach spaces?

\end{document}